\documentclass[10pt,reqno]{conm-p-l}
\usepackage{amsmath,amsthm,amscd}
\usepackage{amssymb,latexsym}
\usepackage{enumerate}

\usepackage{graphicx,psfrag}

\usepackage{framed}

\numberwithin{equation}{section}

\def\divides{{\mathchoice{\mathrel{\bigm|}}{\mathrel{\bigm|}}{\mathrel{|}}{\mathrel{|}}}}
\def\notdivides{\mathrel{\kern-3pt\not\!\kern3.5pt\bigm|}}
\def\smallnotdivides{\mathrel{\kern-2pt\not\!\kern3.5pt\vert}}

\newcommand\dirichlet{\mathsf d}
\newcommand{\fix}{\operatorname{\mathsf F}}

\newcommand{\orbit}{\operatorname{\mathsf O}}
\newcommand{\mertens}{\operatorname{\mathsf M}}
\newcommand{\aut}{\operatorname{Aut}}
\newcommand{\ord}{\operatorname{ord}}
\newcommand{\module}{\operatorname{mod}}

\DeclareMathOperator{\genlin}{GL}

\theoremstyle{plain}
\newtheorem{theorem}{Theorem}
\newtheorem{lemma}[theorem]{Lemma}
\newtheorem{corollary}[theorem]{Corollary}
\newtheorem{proposition}[theorem]{Proposition}%

\theoremstyle{definition}
\newtheorem{definition}[theorem]{Definition}

\theoremstyle{remark}
\newtheorem{example}[theorem]{Example}

\newtheorem{problem}{Problem}

\renewcommand{\le}{\leqslant}
\renewcommand{\ge}{\geqslant}
\renewcommand{\epsilon}{\varepsilon}

\def\bigo{\operatorname{O}}    

\def\T{\mathbb T}
\def\({\left(}\def\){\right)}

\begin{document}

\title{Dynamical invariants for group automorphisms}

\author{Richard Miles}
\address{[RM] School of Mathematics, University of East Anglia, Norwich NR4 7TJ, UK}

\author{Matthew Staines}
\address{[MS] School of Mathematics, University of East Anglia, Norwich NR4 7TJ, UK}

\author{Thomas Ward}
\address{[TW] Department of Mathematical Sciences, Durham University, DH1 3LE, UK}
\email{t.b.ward@durham.ac.uk}



\begin{abstract}
We discuss some of the issues that arise in attempts to classify
automorphisms of compact abelian groups from a dynamical
point of view. In the particular case of automorphisms of
one-dimensional solenoids, a complete description is given
and the problem of determining the range of certain invariants
of topological conjugacy is discussed. Several new
results and old and new open problems are described.
\end{abstract}

\maketitle

\section{Introduction}
\label{Suffering long time angels enraptured by Blake}

Automorphisms and rotations
of compact abelian metric groups provide the simplest
examples of dynamical systems, and their special structure makes
them amenable to detailed analysis. In particular, their strong
homogeneity properties allow global invariants (notably the
topological entropy) to be calculated locally.

Here we wish to discuss the problem of classifying
compact group automorphisms viewed
as dynamical
systems in the most naive sense. Denote by~$\mathcal{G}$ the space
of all pairs~$(G,T)$, where~$T$ is an automorphism of
a compact metric abelian group~$G$, and let~$\sim$ denote a dynamically
meaningful notion of equivalence between two such systems.
Then the problem we wish to discuss has three facets:
\begin{itemize}
\item \emph{Classification}: can the quotient
space~$\mathcal{G}/\mathord\sim$ be described?
\item \emph{Range}: for some invariant~$\pi$
of such systems, what is
\[
\{\pi(G,T)\mid(G,T)\in\mathcal{G}\}?
\]
\item \emph{Fibre}: for a given value~$f$ of some invariant
or collection of invariants~$\pi$,
what can be said about
\[
\{(G,T)\in\mathcal{G}\mid\pi(G,T)=f\}?
\]
\end{itemize}
We will describe some of the issues that arise in
formulating these questions, recalling in particular some
of the well-known difficulties in the classification problem
both from an ergodic theory and a set theory
point of view.
For the range problem we will describe some recent work and
present some new arguments concerned with orbit growth in
particular. Many but not all of the problems identified here are
well-known.

The same questions make sense for group actions. Given a countable
group~$\Gamma$, there is an associated space~$\mathcal{G}_{\Gamma}$
comprising all pairs~$(G,T)$, where~$T:\Gamma\to\aut(G)$ is
a representation of~$\Gamma$ by automorphisms of the
compact metric abelian group~$G$. While our emphasis
is not on any setting beyond~$G=\mathbb Z$, one or two
examples will be given to indicate where the theory
has new features for other group actions.

An element~$(G,T)\in\mathcal{G}$ (by assumption) carries
a rotation-invariant metric giving the
topology, a~$\sigma$-algebra~$\mathcal{B}_G$ of Borel sets,
and a probability measure (Haar measure)~$m_G$ preserved by~$T$.
The most natural notion of equivalence
between~$(G_1,T_1)$ and~$(G_2,T_2)$
in~$\mathcal{G}$
is a commutative diagram
\[
\begin{CD}\label{...}
G_1@>T_1>>G_1\\
@V\phi VV @ VV\phi V\\
G_2@>>T_2>G_2\\
\end{CD}
\]
where the equivariant map~$\phi$ is required to be:
\begin{itemize}
\item a continuous isomorphism of groups, giving the equivalence of \emph{algebraic isomorphism};
\item a homeomorphism, giving the equivalence of \emph{topological conjugacy}; or
\item an almost-everywhere defined isomorphism between the measure spaces
\[
(G_1,\mathcal{B}_{G_1},m_{G_1})
\]
and
\[
(G_2,\mathcal{B}_{G_2},m_{G_2}),
\]
giving the equivalence of
\emph{measurable isomorphism}.
\end{itemize}
Clearly algebraic isomorphism implies topological conjugacy.
Entropy arguments show that topological conjugacy implies
measurable isomorphism. An easy consequence of the type of constructions
discussed later --- and an instance of the type of question one
might ask about the structure of~$\mathcal{G}/\mathord\sim$ --- is that each measurable
isomorphism equivalence class splits into uncountably many topological
conjugacy classes.
We have not included the
equivalence relation that is in some ways the most natural in
dynamics (\emph{finitary isomorphism}, essentially
measurable isomorphism by a map that is continuous off an
invariant
null set) because there is little that can be said about it
beyond those examples where it is almost self-evident how
to construct such a map.

We will also be interested in various invariants for~$(G,T)\in\mathcal{G}$.
The most significant of these is the entropy,
\begin{equation}\label{equation:entropy}
h(T)=\lim_{\epsilon\searrow0}\lim_{n\to\infty}
-\frac{1}{n}\log m_{G}\left(
\bigcap_{j=0}^{n-1}T^{-j}B_{\epsilon}(0)
\right)
\end{equation}
where~$B_{\epsilon}(0)$ denotes an open metric ball around the
identity of~$G$. A manifestation of the homogeneity of the
dynamics of group automorphisms is that this quantity
coincides with the topological entropy and with the
measure-theoretic entropy with respect to Haar measure (see Bowen~\cite{MR0274707}
or~\cite{ELW}), and is therefore an invariant of
any of the notions of equivalence above.

Topological conjugacy also preserves closed orbits
and periodic points. This combinatorial data is all contained in the
dynamical zeta function
\begin{equation}\label{equation:zetafunction}
\zeta_{T}(z)=\exp\sum_{n\ge1}\frac{\fix_T(n)}{n}{z^n},
\end{equation}
where~$\fix_T(n)=\vert\{g\in G\mid T^ng=g\}\vert$,
with radius of convergence
\[
1/\limsup_{n\to\infty}\fix_T(n)^{1/n}.
\]

\begin{definition}\label{definition:handzetaequivalence}
For elements~$(G,T),(G',T')\in\mathcal{G}$ write
\[
(G,T)\sim_{h}(G',T')
\]
if~$h(T)=h(T')$, and
\[
(G,T)\sim_{\zeta}(G',T')
\]
if~$\zeta_{T}(z)=\zeta_{T'}(z)$
as formal power series.
\end{definition}

The primary growth rate measures are
\[
p(T)=\lim_{n\to\infty}\frac{1}{n}\log\fix_T(n),
\]
if this exists (we
write~$p^{+}(T)$ and~$p^{-}(T)$
for the upper and lower limits in general)
and
\[
\pi_T(X)=\sum_{n\le X}\orbit_T(n),
\]
where~$\orbit_T(n)$ denotes the number of closed
orbits of length~$n$.
As we will see, for many (in
several precise senses, most) examples
the limit defining~$p$ will not exist, and a more averaged measure of
orbit growth comes from the dynamical Mertens' theorem,
giving asymptotics for the quantity
\[
\mertens_T(N)=\sum_{n\le N}\frac{\orbit_T(n)}{{\rm e}^{h(T)n}}
=\sum_{n\le N}\sum_{d\divides n}\frac{\mu(n/d)\fix_T(d)}{n{\rm e}^{h(T)n}},
\]
where the expression on the right is simply obtained by
M\"obius inversion, as
\[
\fix_T(n)=\sum_{d\divides n}d\orbit_T(n).
\]

In general, for an action~$T$ of a countable group~$\Gamma$,
the periodic point data for the action is a map
from the space of finite-index subgroups of~$\Gamma$
to~$\mathbb N\cup\{\infty\}$,
associating to each subgroup~$L<\Gamma$
the cardinality of the set of~$L$-periodic points
\begin{equation}\label{We skipped the light fandango}
\fix_{T}(L)=\{x\in X\mid T_{\ell}x=x\mbox{ for all }\ell\in L\}.
\end{equation}
There is an extension of the dynamical zeta function
to the setting of~$\mathbb Z^d$-actions due to Lind~\cite{MR1411232}, but
it does not carry all the periodic point data
unless~$d=1$. Some of the questions considered
here have been studied for full-shift actions
of nilpotent groups in~\cite{MR2465676},
and for specific examples of~$\mathbb Z^2$-actions
by automorphisms of compact connected groups
in~\cite{MR2650793}.

Several of the questions we will discuss
involve adopting a point of view on what is a
natural family of dynamical systems, and what
is a dynamically meaningful notion of equivalence.
In this context
it is reasonable to restrict attention to ergodic
group automorphisms throughout, and in some cases
the more interesting questions arise if we
also insist that the group be connected.
From this point of view the natural starting point
is to study automorphisms of one-dimensional
solenoids, and much of what we will discuss
will be concerned with these. A remarkable consequence
of the work of Markus and Meyer is that
all of these seemingly exotic
one-dimensional
solenoids appear generically in Hamiltonian flows on compact
symplectic manifolds of
sufficiently high dimension~\cite{MR556887}
(we thank Alex Clark for pointing this out). Thus,
despite the seemingly
strange phenomena and
arithmetic questions that arise here, the dynamical
systems discussed here arise naturally in smooth
dynamical situations. Another reason to restrict
much of our attention to one-dimensional
solenoids is cowardice. In one natural direction of
extension, the $S$-integer construction builds
families of compact group automorphisms starting
with algebraic number fields or function fields
of positive characteristic (see~\cite{MR1461206} for the
details), and the arithmetic questions
that arise become more intricate but are
in principle amenable
to the same methods in those settings. Another
natural direction is to simply replace `one'
by~$d\ge2$, and study automorphisms of~$d$-dimensional
solenoids. In this complete generality the difficulties
are more formidable, and this is discussed
further at the end of Section~\ref{subsectionsubgroupsoftherationals}.

The assumption of ergodicity is a mild requirement
and avoids degeneracies -- automorphisms of compact abelian groups
cannot be non-ergodic in dynamically interesting ways.
Ergodicity guarantees (completely) positive entropy,
and the structure of compact group automorphisms
with zero entropy has been described by Seethoff~\cite{MR2617673}.

\section{Algebraic isomorphism}

Deciding if~$(G_1,T_1)$ and~$(G_2,T_2)$ are algebraically
isomorphic is not a dynamical question at all. The equivalence
implies that~$\phi:G_1\to G_2$ is an isomorphism of groups, so
we are asking if the groups are isomorphic to a single
group~$G$, and then if the two elements~$T_1$
and~$\phi^{-1}T_2\phi$ are conjugate as elements of the
group~$\aut(G)$ of continuous automorphisms of~$G$. If the
systems are expansive (or satisfy the descending chain
condition) and connected, then the problem of determining
conjugacy in the automorphism group is algorithmically
decidable by a result of Grunewald and Segal~\cite{MR805682}
(see Kitchens and Schmidt~\cite[Sec.~6]{MR1036904}). In general
this cannot be expected: if the group is infinite-dimensional
then elements of the automorphism group in general correspond
to matrices of infinite rank.

\begin{example}\label{example:torus}
Toral automorphisms are conjugate if the
corresponding matrices have the same characteristic polynomial
and share an additional number-theoretic invariant
related to ideal classes in the splitting field
of the characteristic polynomial (see
Latimer and Macduffee~\cite{MR1503108} for
the details). This additional invariant is
already visible even on the~$2$-torus. For example,
if
\[
A=\begin{pmatrix}3&10\\1&3\end{pmatrix},
B=\begin{pmatrix}3&5\\2&3\end{pmatrix}
\]
then it is easy to check that if~$Q\in{\rm M}_{2}(\mathbb Z)$
has~$QA=BQ$ we have~$5\divides\det Q$.
\end{example}

A flavour of the sort of examples that will arise in this
setting is given by the next problem.

\begin{problem}\label{splittingdiagram}
Given a monic polynomial~$\chi\in\mathbb Z[x]$ of degree~$d$, describe
the combinatorial properties of the following
poset. At the bottom level~$n=0$ there is a
vertex for each
of the finitely many
algebraic
conjugacy classes of the
set
\[
M_{\chi}=\{A\in{\rm GL}_d(\mathbb Z)\mid\det(tI_d-A)=\chi(t)\},
\]
which may be enumerated using data from the ideal
class groups of the splitting fields of the
various irreducible factors of~$\chi$ using~\cite{MR1503108}.
At level~$n\ge1$ there is a
vertex for each of
the conjugacy classes of the elements of~$M_{\chi}$
over the ring~$\mathbb Z[\frac{1}{p_j}:j\le n]$ where~$p_1,p_2,\dots$
are the rational primes in their natural order.
The edges reflect conjugacy classes merging
in the larger group at each level, and the top of
the poset has a vertex for each conjugacy class
of the set~$M_{\chi}$ in~${\rm GL}_d(\mathbb Q)$,
which are described by the rational canonical
forms of matrices with characteristic polynomial~$\chi$.
\end{problem}

For example, Problem~\ref{splittingdiagram} is readily solved
for Example~\ref{example:torus} as follows. The
polynomial~$\chi(x)=x^2-6x-1$ is irreducible over~$\mathbb Q$
with splitting field~$\mathbb Q(\sqrt{10})$, with ideal class
group isomorphic to~$\mathbb Z/2\mathbb Z$. The representative
matrices in Example~\ref{example:torus} are not conjugate
over~$\mathbb Z[\frac12]$ and~$\mathbb Z[\frac16]$, but become
conjugate over~$\mathbb Z[\frac{1}{30}]$, resulting in the
diagram in Figure~\ref{figure:splitting}.

\begin{figure}[htbp]\begin{center}
\psfrag{0}{$n=0$}\psfrag{1}{$1$}\psfrag{2}{$2$}\psfrag{3}{$3$}\psfrag{4}{$4$}
\scalebox{1.0}[1.0]{\includegraphics{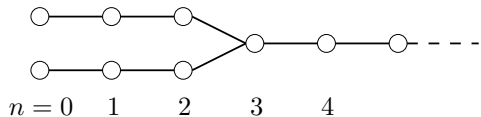}}
\caption{\label{figure:splitting}Conjugacy for the
characteristic polynomial $x^2-6x-1$.}
\end{center}
\end{figure}

\section{Topological conjugacy}

As we will see, for connected groups this notion of
equivalence collapses to algebraic isomorphism. In fact
there is a rigidity phenomenon at work, which means
that in many situations not only does the existence of
a topological conjugacy force there to be an algebraic
isomorphism, it is required to be algebraic itself.

\begin{theorem}[Adler and Palais~\cite{MR0193181}]\label{theoremadlerpalais}
If~$T:\T^d\to\T^d$ and~$S:G\to G$ are topologically
conjugate ergodic automorphisms via a
homeomorphism~$\phi$ from~$\T^d$ to~$G$, then~$G\cong\T^d$ and the
map~$\phi$ is itself an automorphism of~$\T^d$ composed
with rotation by a fixed point of~$S$.
\end{theorem}

This extends in several ways to other connected
finite-dimensional groups, and to other group
and semi-group actions (see Clark and Fokkink~\cite{MR2308207}
and Bhattacharya and the third author~\cite{MR2129101}). Notice that
the phenomena is not universal, and in particular
requires the topological entropy to be finite.
The next example shows that in infinite entropy situations
there is no topological rigidity.

\begin{example}
If~$\theta:\mathbb T\rightarrow\mathbb T$ is any
homeomorphism, then the map~$\Theta:\mathbb T^{\mathbb Z}\rightarrow
\mathbb T^{\mathbb Z}$ defined by~$(\Theta\((x_n)_{n\in\mathbb Z}\))_k=\theta(x_k)$
for all~$k\in\mathbb Z$ is a topological conjugacy between
the left shift on~$\mathbb T^{\mathbb Z}$ and itself.
\end{example}

\begin{example}\label{exampletopzerodimensional}
In contrast to the topological
rigidity in Theorem~\ref{theoremadlerpalais}
(as pointed out by Kitchens and Schmidt~\cite{MR1036904})
topological conjugacy on zero-dimensional groups
preserves very little of the algebraic structure.
\begin{enumerate}
\item[\rm(1)]
A finite abelian group~$G$ has an associated shift
automorphism~$(G^{\mathbb Z},\sigma_G)$ where~$\(\sigma((x_n))\)_{k}=x_{k+1}$
for all~$k\in\mathbb Z$. Clearly~$(G^{\mathbb Z},\sigma_G)$
will be topologically conjugate to~$(H^{\mathbb Z},\sigma_H)$
if and only if~$\vert G\vert=\vert H\vert$.
Thus no part of the internal structure of the group
alphabet is preserved by topological conjugacy.
\item[\rm(2)]
Kitchens~\cite{MR896796} proves that if~$G$ is zero-dimensional
and~$T$ is expansive, meaning that there is an open neighbourhood~$U$
of the identity with
\[
\bigcap_{n\in\mathbb Z}T^{-n}U=\{e\},
\]
and has a dense orbit, then~$(G,T)$ is topologically conjugate to a
full shift on~${\rm e}^{h(T)}$ symbols.
In the same paper the problem of classifying
such automorphisms up to simultaneous algebraic
isomorphism and topological conjugacy is also addressed
but not completely solved. It is shown that for any given
entropy there are only finitely many equivalence classes
for simultaneous algebraic isomorphism and algebraic
conjugacy.
If the entropy is~$\log p$ for some prime~$p$ then
it is shown that there is only one equivalence class,
namely that containing~$(\mathbb F_p^{\mathbb Z},\text{shift})$.
\item[\rm(3)]
For other group actions entirely new rigidity phenomena appear.
Given a finite abelian group~$G$, we may define an
action~$\sigma_G$ of~$\mathbb Z^2$ on the compact group
\[
X_G=\{x=(x_{n,m})\in G^{\mathbb Z^2}\mid
x_{n,m}+x_{n+1,m}=x_{n,m+1}\mbox{ for all }n,m\in\mathbb Z\}
\]
by shifting. R{\"o}ttger~\cite{MR2141759},
extending a partial result in~\cite{MR1631042}, shows that the periodic
point data in the sense of~\eqref{We skipped the light fandango}
for the system determines the structure of~$G$, and so
in particular if~$(X_G,\sigma_G)$ is topologically
conjugate to~$(X_H,\sigma_H)$ then~$G\cong H$.
\item[\rm(4)]
For the case~$G=\mathbb Z/2\mathbb Z$ in the setting of~$X_G$
as in~(3), the only homeomorphisms~$X_G\to X_G$ commuting
with the action of~$\sigma_G$ are elements of the action itself
by~\cite{MR1307966}.
\end{enumerate}
\end{example}

\begin{problem}[Kitchens~\cite{MR896796}]
Classify expansive ergodic automorphisms of compact
abelian zero-dimensional groups up to simultaneous
algebraic isomorphism and topological conjugacy.
\end{problem}

Example~\ref{exampletopzerodimensional}(3) is associated in a
natural way to the polynomial
\[
1+u_1-u_2\in\mathbb Z[u_1^{\pm1},u_2^{\pm1}]
\]
since the condition defining the elements of~$X_G$
may be thought of as requiring that the element of~$G^{\mathbb Z^2}$
is annihilated by~$1+u_1-u_2$, with~$u_1$ corresponding to a shift
in the first coordinate and~$u_2$ to a shift in the second
coordinate.

\begin{problem}
Associate to a polynomial~$f\in\mathbb Z[u_1^{\pm1},u_2^{\pm1}]$
and a finite abelian group~$G$ the shift action~$\sigma_{G,f}$
on the space~$X_{G,f}$ of elements of~$G^{\mathbb Z^2}$
annihilated by~$f$. Does the periodic point data
for~$\sigma_{G,f}$ determine the structure of~$G$?
\end{problem}

\section{Measurable isomorphism}

This is an opaque equivalence relation,
because as measure spaces any two infinite compact
metric abelian groups are isomorphic. A consequence
of the homogeneity mentioned above is that
the entropy~$h(T)$ is an invariant of measurable
isomorphism. A much deeper fact is that there is
an abstract model
as an independent identically distributed
stochastic process
for any ergodic group automorphism.
These independent identically distributed processes -- Bernoulli shifts -- in turn are
classified up to measurable isomorphism by their entropy (see Theorem~\ref{theoremornstein} below).

\begin{example}
Example~\ref{exampletopzerodimensional} gave some
instances of topological rigidity for actions of
larger groups. There is a remarkable theory of
rigidity for measurable isomorphism of algebraic~$\mathbb Z^d$
actions, and we refer to the survey of Schmidt~\cite{MR1858549}
and the references therein for this.
\end{example}

\begin{theorem}\label{theorem:groupautomorphismsandBernoullishifts}
Let~$(G,T)\in\mathcal{G}$ be an ergodic
compact group automorphism. Then there is
a probability vector~$(p_1,p_2,\dots)$
defining a measure~$\mu_p$ on~$\mathbb N$
with the property that~$(G,T)$ is measurably
isomorphic to the shift map~$\sigma$ on the
space~$\mathbb N^{\mathbb Z}$ preserving the
product measure~$\prod_{n\in\mathbb Z}\mu_p.$
Moreover,~$h(T)$ is equal to the measure-theoretic
entropy~$\sum_{n\in\mathbb N}p_n\log p_n$
of~$\sigma$ with respect to this measure. If~$h(T)<\infty$
then we may assume that~$p_k=0$ for all but finitely many~$k$,
or equivalently we may replace the alphabet~$\mathbb N$
with a finite set.
\end{theorem}

This was proved by Katznelson~\cite{MR0294602} in the case~$G=\mathbb T^d$,
and shown in general by Lind~\cite{MR0460593}
and independently by Miles and Thomas~\cite{MR517262} and
Aoki~\cite{MR605377}, using abstract characterizations of the
property of being measurably isomorphic to a Bernoulli
shift developed by Ornstein~\cite{MR0447525}.
It means that an ergodic compact
group automorphism is measurably indistinguishable from
an abstract independent identically distributed stochastic
process. Notice that the space~$A^{\mathbb Z}$ for a
finite alphabet~$A$ is a compact metric space in
a natural way, so the question of whether or not the
measurable isomorphism in Theorem~\ref{theorem:groupautomorphismsandBernoullishifts}
may be chosen to be finitary -- continuous off a null set --
in the case of finite entropy
arises. Little is known in general, but Smorodinsky (unpublished)
pointed out that a necessary condition
is exponential recurrence -- for a given non-empty open
set~$O$, the Haar measure of the set of points
first returning to~$O$ after~$n$ iterations of the
automorphism decays exponentially in~$n$. Lind~\cite{MR657863}
showed that any ergodic group automorphism is
exponentially recurrent, so all that is known is that
this particular property does not preclude the possibility
that ergodic group automorphisms are finitarily isomorphic to
Bernoulli shifts. Rudolph has given a characterization of
the property of being finitarily isomorphic to
a Bernoulli shift~\cite{MR633760}.

\begin{problem}
Determine when an ergodic group automorphism is
finitarily isomorphic to a Bernoulli shift.
\end{problem}

\begin{theorem}[Ornstein~\cite{MR0257322}]\label{theoremornstein}
Bernoulli shifts with the same entropy are
measurably isomorphic.
\end{theorem}

Theorems~\ref{theorem:groupautomorphismsandBernoullishifts}
and~\ref{theoremornstein} together mean that the
space~$\mathcal{G}/\mathord\sim$ for the equivalence of
measurable isomorphism embeds into~$\mathbb R_{>0}$. There is a
representative Bernoulli shift for each measurable isomorphism
class of compact group automorphisms -- but is there a compact
group automorphism for each Bernoulli shift? Equivalently, is
there a compact group automorphism for each possible entropy?

In order to describe what is known about this
problem, let~$f\in\mathbb Z[x]$ be a polynomial
with integer coefficients, with factorization~$f(z)=\prod_{i=1}^{d}(z-\lambda_i)$
over~$\mathbb C$. Then the \emph{logarithmic
Mahler measure of~$f$} is the quantity
\[
m(f)=\int_0^1\log\vert f({\rm e}^{2\pi{\rm i}s})\vert{\rm d}s
=\sum_{\vert\lambda_i\vert>1}\log\vert\lambda_i\vert=\sum_{i=1}^{d}\log^{+}\vert\lambda_i\vert.
\]
Thus Kronecker's lemma is the statement that~$m(f)=0$ if and only if~$f$
is a product of cyclotomic polynomials.
Lehmer~\cite{MR1503118} raised the question of whether
there could be small positive values of the Mahler
measure.
He found the smallest value
known to date, corresponding to
the polynomial
\[
f(x)=x^{10}+x^9-x^7-x^6-x^5-x^4-x^3+x+1
\]
with Mahler measure approximately~$\log1.176$.
Write
\[
L=\inf\{m(f)\mid m(f)>0\}.
\]

Yuzvinski{\u\i}~\cite{MR0214726} computed the entropy
of endomorphisms of solenoids (see Lind and
the third author~\cite{MR961739}
for the history of earlier results and a
simpler approach),
and this led to a complete description
of the range of possible values.
These relate to the dynamics of group automorphisms
in several ways.

\begin{theorem}[Lind~\cite{MR0346130}]
The infinite torus~$\mathbb T^{\infty}$ has ergodic
automorphisms with finite entropy if
and only if~$L=0$.
\end{theorem}

An expansive automorphism is guaranteed to have finite
entropy, but unfortunately
Hastings~\cite{MR518555} shows that no automorphisms of
the infinite torus can be expansive.

\begin{theorem}[Lind~\cite{MR0460593}]\label{dougstheorem}
The set of possible entropies of ergodic
automorphisms of compact groups is~$(0,\infty]$
if~$L=0$, and is the countable set
\[
\{m(f)\mid f\in\mathbb Z[x],m(f)>0\}
\]
if~$L>0$.
\end{theorem}

This means that the cardinality of the quotient space~$\mathcal{G}/\mathord\sim$
is uncountable or countable depending on the solution to Lehmer's problem.
This is now a problem of some antiquity, and means that we do not
know if group automorphisms are special or not when
viewed
as measure-preserving transformations
among the systems measurably isomorphic to Bernoulli
shifts.

\begin{problem}[Lehmer~\cite{MR1503118}]
Determine if~$L=\inf\{m(f)\mid m(f)>0\}>0$.
\end{problem}

We refer to Boyd~\cite{MR644535} for an overview of Lehmer's
problem. Expressed in dynamical systems,
Theorem~\ref{dougstheorem} may be used to express the same
question as follows.

\begin{problem}[Lind~\cite{MR0460593}]
Determine if it is possible to
construct a compact group automorphism whose
entropy is not of the form~$m(f)$ for some~$f\in\mathbb Z[x]$.
\end{problem}

Attempting to describe the fibres (that is, the equivalence
classes) for measurable isomorphism requires answering an
already difficult problem in number theory, the \emph{inverse
Mahler measure problem}.

\begin{problem}
Given~$f\in\mathbb Z[x]$, characterize the set~$\{g\in\mathbb Z[x]\mid
m(g)=m(f)\}$.
\end{problem}

We refer to Boyd~\cite{MR874125} for an overview of this problem,
and to Staines~\cite{staines} for recent work and
references.

To conclude this section, it is worth noting that for other
group actions the same entropy range problem, that is, the
problem of determining the set
\begin{equation}\label{entropy_range}
\{h(T)\mid (G,T)\in\mathcal{G}_\Gamma\}
\subset\mathbb R,
\end{equation}
can turn out to be considerably simpler than providing an
answer to Lehmer's problem corresponding to the
case~$\Gamma=\mathbb{Z}$. For
example, Bj{\"o}rklund and Miles~\cite[Th.~4.1]{MR2533986}
show that if~$\Gamma$ is
any countably infinite amenable group with arbitrarily large
finite normal subgroups, then
the entropy range~\eqref{entropy_range}
is~$[0,\infty]$. In contrast, for~$\Gamma=\mathbb{Z}^d$, the
entropy range problem again reduces to Lehmer's problem
by~\cite{MR1062797}. Further discussion of the entropy range
problem for group actions may be found in~\cite{MR2533986}.

\section{One-solenoids and subgroups of the rationals}
\label{section:One-solenoids and subgroups of the rationals}

Some of the classification and range problems we wish to
discuss are already interesting in the special case
of one-dimensional solenoids. A compact metric abelian
group is called a one-solenoid if it is connected and
has topological dimension one. Equivalently, a one-solenoid
is a group whose dual or character group is a
subgroup of~$\mathbb Q$. In this section we
briefly recall the well-known description of these
subgroups.

\subsection{Subgroups of the rationals}\label{subsectionsubgroupsoftherationals}

Subgroups of~$\mathbb Q$ are readily classified (see
Baer~\cite{MR1545974}, or Beaumont and Zuckerman~\cite{MR0044522}
for a modern account).
Let~$H\le\mathbb Q$ be an additive subgroup,
and~$x\in H\setminus\{0\}$. Write~$\mathbb P=\{2,3,5,7,\dots\}$ for the set of rational primes. For each prime~$p\in\mathbb P$,
the \emph{$p$-height} of~$x$
is
\[
k_p(x)=\sup\{n\in\mathbb N\mid p^ny=x\mbox{ has a solution }y\in H\}
\in\mathbb N\cup\{\infty\}
\]
and the \emph{characteristic} of~$x$
is the sequence
\[
k(x)=\(k_p(x)\)_{p\in\mathbb P}\in\(\mathbb N\cup\{\infty\}\)^{\mathbb N}.
\]
Two sequences~$(a_p),(b_p)\in\(\mathbb N\cup\{\infty\}\)^{\mathbb N}$
\emph{belong to the same type} if~$a(p)=b(p)$ for all
but finitely many~$p$, and
for
any~$p$ with~$a(p)\neq b(p)$, both~$a(p)$ and~$b(p)$
are finite.
Notice that if~$x,y\in H\setminus\{0\}$ then~$k(x)$
and~$k(y)$ belong to the same type,
allowing us to define the \emph{type of~$H$},~$k(H)$, to
be the type of any non-zero element of~$H$.

On the other hand, given any
sequence~$(k_p)_{p\in\mathbb P}$ with~$k_p\in\mathbb N\cup\{\infty\}$,
we may define
\[
H\((k_p)_{p\in\mathbb P}\)=
\{\textstyle\frac{a}{b}\in\mathbb Q\mid
a,b\in\mathbb Z,\gcd(a,b)=1,
\ord_{p}(b)\le k_p\mbox{ for all }p\in\mathbb P\},
\]
an additive subgroup of~$\mathbb Q$.

\begin{theorem}[Baer~\cite{MR1545974}]\label{theorem:baer}
Any subgroup of~$\mathbb Q$ is of the form~$H\((k_p)_{p\in\mathbb P}\)$
for some sequence~$(k_p)\in\(\mathbb N\cup\{\infty\}\)^{\mathbb N}$.
Two subgroups of~$\mathbb Q$ are isomorphic
if and only if they are of the same type.
\end{theorem}

Clearly the only self-homomorphisms of~$H=H\((k_p)\)$
are the maps~$x\mapsto\frac{a}{b}x$ with~$\gcd(a,b)=1$
and~$k_p(H)=\infty$ for any prime~$p$ dividing~$b$.
This gives a description of all
continuous automorphisms
of one-solenoids:~$(G,T)\in\mathcal{G}$ has~$G$
connected with topological dimension~$1$
if and only if~$G$ is dual to a group~$H\((k_p)_{p\in\mathbb P}\)$
and~$T$ is dual to a map~$x\mapsto\frac{a}{b}x$
with~$\gcd(a,b)=1$ and
with~$k_p=\infty$ for any prime~$p$ dividing~$a$ or~$b$.
Writing~$\mathcal S(H)$ for the
set of primes~$p$ with~$k_p(H)=\infty$,
Theorem~\ref{theorem:baer} shows in particular that~$G\cong G'$ implies that~$\mathcal S(\widehat{G})=\mathcal S(\widehat{G}')$.

\begin{example}\label{baerexamples}
Some of the diversity of subgroups
may be seen in the following examples.
\begin{enumerate}
\item $H\((\infty)_{p\in\mathbb P}\)=\mathbb Q$.
\item $H\((0)_{p\in\mathbb P}\)=\mathbb Z$.
\item $H\((\infty,\infty,0,0,\dots)\)=\mathbb Z[\frac16]$.
\item $H\((0,1,1,1,1,\dots)\)$ is the subgroup of
rationals with odd square-free denominator.
\item\label{s_integer_definition} For a set~$\mathcal S\subset\mathbb P$ of primes, there
is an associated subgroup of~$\mathcal S$-integers
\[
R_\mathcal S=\{x\in\mathbb Q\mid
\vert x\vert_p\le1\mbox{ for all }p\notin \mathcal S\}
=\mathbb Z[\textstyle\frac{1}{p}:p\in \mathcal S]
=H\((k_p)_{p\in\mathbb P}
\)
\]
where~$k_p=\infty$ if~$p\in S$ and~$0$ otherwise.
\item Let~$\omega=(\omega_j)\in\{H,T\}^{\mathbb N}$ be the outcome of an
infinitely repeated fair coin toss, and let
\[
k_{p_j}=\begin{cases}\infty&\mbox{if }j=H;\\0&\mbox{if }j=T.\end{cases}
\]
Then~$H\((k_p)_{p\in\mathbb P}\)$ is a `random'~$\mathcal S$-integer subgroup.
\end{enumerate}
\end{example}

Solenoids of higher topological dimension,
dual to subgroups of~$\mathbb Q^d$ with~$d$ greater than~$1$,
present peculiar difficulties unless there are additional
assumptions. A solenoid that carries an expansive
automorphism has a prescribed structure (Lawton~\cite{MR2622030},~\cite{MR0391051};
see also Kitchens and Schmidt~\cite[Sec.~5]{MR1036904} for
this result in the context of a more general treatment
of finiteness conditions on group automorphisms)
rendering it amenable to analysis; solenoids dual
to rings of~$\mathcal{S}$-integers in number fields also have
a prescribed structure, and this is used by Chothi {\it et al.}
to study orbit growth in these systems~\cite{MR1461206},~\cite{MR2339472}
and~\cite{MR2550149}.
Around the same time as Baer's classification of
subgroups of~$\mathbb Q$, Kurosh~\cite{MR1503333} and
Malcev~\cite{64.0064.05} described systems of
complete invariants for subgroups of~$\mathbb Q^d$,
but determining when the invariants
correspond to isomorphic groups is intractable for~$d>1$.
The difficulties encountered here now have a precise
formulation in terms of descriptive set theory,
and Thomas~\cite{MR1937205} proves
the remarkable result that the classification
problem for subgroups of~$\mathbb Q^{n+1}$ is
strictly more difficult than the same problem for~$\mathbb Q^n$
for~$n\ge1$ (we refer to the survey by Thomas~\cite{MR2275590}
for an overview of these results).

\subsection{Fixed points and entropy on one-solenoids}

There are several routes to a formula for~$\fix_T(n)$ if~$T$
is an automorphism of a one-solenoid: a group-theoretic
argument due to England and Smith~\cite{MR0307280},
a general valuation-theoretic approach due to Miles~\cite{MR2441142}
which is independent of the characteristic and the topological
dimension,
and in the~$\mathcal{S}$-integer case an adelic
argument due to Chothi, Everest {\it et al.}~\cite{MR1461206}.
Here we adapt the adelic approach
of Weil~\cite{MR0234930}
to cover all subgroups,
because this geometrical approach also evaluates
the entropy (the harmonic analysis used here is
not strictly adelic, but closer to the constructions
in Tate's thesis~\cite{MR2612222}, reproduced as~\cite{MR0217026}).
For completeness we explain in Lemma~\ref{lemma:modules}
the basic
idea from~\cite[Lem.~5.1]{MR1461206},
and refer to Hewitt and Ross~\cite{MR551496}
for background on the duality theory
between discrete abelian groups and compact
abelian groups.

The dual of a one-solenoid may be identified with a subgroup
of~$\mathbb Q$, and we now describe such a group using a direct
limit of fractional ideals over a ring
of~$\mathcal{S}$-integers (see
Example~\ref{baerexamples}(\ref{s_integer_definition})) based
on the concept of types introduced in the previous section.
Let~$R_\mathcal{S}$ be a ring of~$\mathcal S$-integers. A set
of rational primes~$\mathcal{P}$ disjoint from~$\mathcal{S}$
corresponds to a set of non-zero prime ideals
of~$R_\mathcal{S}$. Let~$k=(k_p)_{p\in\mathcal{P}}$ be a
sequence of positive integers and define a submodule of
the~$R_\mathcal{S}$-module $\mathbb{Q}$ by~$H=\sum_Q c_Q
R_\mathcal S$, where~$c_Q=\prod_{p\in Q}p^{-k_p}$ and~$Q$ runs
over all finite subsets of~$\mathcal P$. The group~$H$ is
isomorphic to~$R_\mathcal S$ if and only if~$\mathcal{P}$ is
finite, otherwise~$H$ may be expressed using a direct
limit~$\injlim_{i\to\infty}H_i$ of~$R_{\mathcal S}$-modules,
where
\[
H_i=p(1)^{-k_{p(1)}}\cdot\dots\cdot p(i)^{-k_{p(i)}}R_{\mathcal S},
\]
and~$p(1),p(2),\dots$ is an enumeration of~$\mathcal P$. By
Theorem~\ref{theorem:baer}, for any given type we may choose a
representative subgroup $H=H ((k_p)_{p\in\mathbb P})$ and note
that the set
\[
\mathcal S(H)=\{p\in\mathbb P\mid k_p=\infty\}
\]
is independent of the choice of representative. We set
\[
\mathcal P(H)=\{p\in\mathbb P\mid 0<k_p<\infty\}
\]
and
\[
k(H)=(k_p)_{p\in \mathcal{P}(H)},
\]
and use this data as above to obtain an
$R_{\mathcal{S}(H)}$-module having the same type as~$H$. Any
other choice of representative subgroup produces an isomorphic
$R_{\mathcal{S}(H)}$-module. This description now gives us a
canonical way to view a one-solenoid.

Since there is a short exact sequence
of discrete groups
\[
\{0\}
\longrightarrow
\mathbb Z
\longrightarrow
R_{\mathcal{S}(H)}
\longrightarrow
\sum_{p\in\mathcal{S}(H)}\mathbb{Z}[\textstyle\frac{1}{p}]/\mathbb Z
\longrightarrow
\{0\},
\]
and the
dual group of~$\mathbb{Z}[\frac{1}{p}]/\mathbb Z$
is isomorphic to the ring of $p$-adic integers~$\mathbb Z_p$, via duality, there is a short exact sequence of compact groups
\begin{equation}\label{Tu pure, o, Principessa}
\{0\}
\longrightarrow
\prod_{p\in \mathcal{S}(H)}
\mathbb Z_p
\longrightarrow
\widehat{R_{\mathcal{S}(H)}}
\longrightarrow
\mathbb T
\longrightarrow
\{0\}.
\end{equation}
Therefore,~$\widehat{R_{\mathcal{S}(H)}}$
is a central extension of~$\mathbb T$
by a cocycle
\begin{equation}\label{nella tua fredda stanza}
w:\mathbb T\times\mathbb T\longrightarrow
\prod_{p\in\mathcal{S}(H)}
\mathbb Z_p,
\end{equation}
and a simple explicit calculation along the lines of~\cite{ward}
shows that we may take
\[
w(s,t)=-\lfloor s+t\rfloor
\((1,1,\dots)\).
\]
Furthermore, there is an explicit geometrical description of
the solenoid~$\widehat{R_{\mathcal{S}(H)}}$
as follows.

\begin{lemma}\label{lemma:adaptedadelicconstruction}
The diagonal map~$\delta:x\mapsto(x,x,\dots)$ embeds~$R_{\mathcal{S}(H)}$
as a discrete and co-compact subgroup of the restricted directed product
\[
\mathbb A_{\mathcal{S}(H)}
=
\left\{
(x_{\infty},(x_p))\in\mathbb R\negmedspace\times\negmedspace\negmedspace
\prod_{p\in\mathcal{S}(H)}\mathbb Q_p
\mid
x_p\in\mathbb Z_p\mbox{ for all but finitely many }p
\right\},
\]
which is a locally compact topological group.
The set
\[
F=[0,1)\times\prod_{p\in\mathcal{S}(H)}\mathbb Z_p
\]
is a fundamental domain for~$\delta(R_{\mathcal{S}(H)})$
in~$\mathbb A_{\mathcal{S}(H)}$,
and
\[
\widehat{R_{\mathcal{S}(H)}}
\cong
\mathbb A_{\mathcal{S}(H)}/
\delta(R_{\mathcal{S}(H)}).
\]
\end{lemma}

\begin{proof}
This is a special case of the construction introduced
in~\cite{MR2612222}. See also~\cite{MR1680912} for an
accessible account.
\end{proof}

Recall that a subgroup~$\Gamma<X$ in a locally compact
topological group is called a \emph{uniform lattice}
if~$\Gamma$ is discrete and the quotient space~$X/\Gamma$ is
compact in the quotient topology, and a measurable
set~$F\subset X$ is a \emph{fundamental domain} for~$\Gamma$ if
it contains exactly one representative of each coset
of~$\Gamma$. The \emph{module} of a homomorphism
\[
A:X\to X
\]
is the quantity~$m(AU)/m(U)$ for an open set~$U$ of finite
measure, where~$m$ is a choice of Haar measure on~$X$.

\begin{lemma}\label{lemma:modules}
Let~$\Gamma$ be a
uniform lattice in a locally compact abelian group~$X$,
let~$F$ be a fundamental domain for~$\Gamma$,
and let~$m$ denote the Haar measure on~$X$
normalized to have~$m(F)=1$.
Let $\widetilde{A}:X\to X$ be a
continuous surjective
homomorphism with~$\widetilde{A}(\Gamma)\subset\Gamma$, and
let~$A:X/{\Gamma}\rightarrow X/{\Gamma}$ be the
induced map.
If~$\ker A$ is discrete,
then
\[
\module_X(\widetilde{A})=\vert\ker A\vert=m(\widetilde{A}F).
\]
\end{lemma}

\begin{proof}
Choose~$F$ so that it contains an open
neighborhood~$U$ of the identity in~$X$ (this
is possible since~$\Gamma$ is discrete in~$X$).
Since~$X/\Gamma$ is compact,~$\ker A$ is finite
and so there is a disjoint union~$A^{-1}V=V_{1}\sqcup\cdots\sqcup V_{\vert\ker A\vert}$
if~$V$ is a sufficiently small neighbourhood
of the identity in~$X/\Gamma$,
with each~$V_i$ an open neighbourhood of a point
in the fibre~$A^{-1}(0_{X/{\Gamma}})$.
Since~$A$ is measure--preserving,
\[
m\left(A^{-1}V\right)=
m\left(V)\right).
\]
If~$U$
and~$V$
are sufficiently small,
then the quotient map~$x\mapsto x+\Gamma$
is a homeomorphism between~$U$ and~$V$, and so
\[
m\left(\widetilde{A}U\right)=m\left(AV\right)
=\vert\ker A\vert m\left(V\right)
=\vert\ker A\vert m\left(U\right);
\]
furthermore, since $U(0_{X})\subset F$,~$m(\widetilde{A}F)=\vert\ker A\vert.$
\end{proof}

\begin{example}\label{fixedpointsexamples}
A simple situation to which Lemma~\ref{lemma:modules} may
be applied is
endomorphisms of the torus.
\begin{enumerate}
\item[\rm(1)] Taking~$\Gamma=\mathbb Z< X=\mathbb R$,~$\widetilde{A}(x)=bx$
for some~$b\in\mathbb Z\setminus\{0\}$, and~$m$ to be Lebesgue measure with~$m\(F=[0,1)\)=1$,
we see that
\[
\vert\ker(x\mapsto bx\pmod 1)\vert=\vert b\vert.
\]
In particular, it follows that if~$T:\mathbb R/\mathbb Z\to\mathbb R/\mathbb Z$
is the map~$x\mapsto ax$
for some~$a$ in~$\mathbb Z\setminus\{0,\pm1\}$
then~$\fix_{T}(n)=\vert a^n-1\vert$ and so
\[
\zeta_{T}(z)=\frac{1-z}{1-az}
\]
if~$a>0$, and
\[
\zeta_{T}(z)=
\exp
\(
\sum_{n\ge1}\frac{\vert a\vert^n}{n}z^n+\sum_{n\ge1}\frac{1}{n}z^{n}-2\sum_{n\ge1}\frac{1}{2n}z^{2n}
\)
=\frac{1+z}{1-\vert a\vert z}
\]
if~$a<0$.
\item[\rm(2)] As seen above, some care is needed in dealing with the
distinction between the expressions~$\vert a^n-1\vert$ and~$\vert a\vert^n-1$:
for an automorphism~$T$ of the~$d$-torus~$G=\mathbb T^d$
defined by an integer matrix~$A_T\in\genlin_d(\mathbb Z)$
the same argument gives the formula
\[
\fix_T(n)=\vert\det(A_T^n-I)\vert,
\]
and {\it a priori} an argument is needed to show that~$\zeta_T(z)$ is a rational
function of~$z$ because of the absolute value. This is
discussed in Smale~\cite[Prop.~4.15]{MR0228014},
and an elementary algorithmic way to compute
the zeta function
in integer arithmetic is given
by Baake, Lau and Paskunas~\cite{MR2670229}.
\end{enumerate}
\end{example}

We can also use Lemma~\ref{lemma:modules}
to find a formula for the periodic points of an
automorphism of a one-solenoid, recovering
in different notation the result of~\cite{MR0307280}
and the one-dimensional case of~\cite{MR2441142}.

\begin{proposition}\label{fixedpointformula}
If~$T:\widehat{H}\to \widehat{H}$ is an automorphism of a one-solenoid dual to the
map~$x\mapsto rx$ on~$H\leqslant\mathbb{Q}$, then
\[
\fix_T(n)=\vert r^n-1\vert\times\prod_{p\in \mathcal{S}(H)}\vert r^n-1\vert_p.
\]
\end{proposition}

\begin{proof}
Since~$x\mapsto rx$ is an automorphism of~$H$,~$p\in\mathcal{S}(H)$
whenever~$\vert r\vert_p\neq1$.
Lemmas~\ref{lemma:adaptedadelicconstruction} and~\ref{lemma:modules}
show that if~$T'$ is the automorphism dual to~$x\mapsto rx$
on~$\mathbb{A}_{\mathcal{S}(H)}$, then
\[
\fix_{T'}(n)=\module_{\mathbb{A}_{\mathcal{S}(H)}}(x\mapsto rx)
=\vert r^n-1\vert\times\prod_{p\in\mathcal{S}(H)}\vert r^n-1\vert_p.
\]
The set of points of period~$n$
is dual to the group~$H/(r^n-1)H$,
and multiplication by~$p^{k_p}$ for any~$p\in\mathcal{P}(H)$
remains non-invertible in the direct
limit~$H$, so
\[
H/(r^n-1)H
\cong
R_{\mathcal{S}(H)}/
(r^n-1)R_{\mathcal{S}(H)}
\]
for each~$n\ge1$,
showing the proposition.
\end{proof}

The geometric viewpoint using adeles also allows the entropy to be computed
easily using~\eqref{equation:entropy};
this calculation is originally due to Abramov~\cite{MR22:8103}.
Write
\[
\log^{+}(t)=\max\{0,\log t\}
\]
for~$t>0$.

\begin{proposition}[Abramov's formula]\label{lemma:abramovformula}
If~$T:\widehat{H}\to \widehat{H}$ is an automorphism of a one-solenoid dual to the
map~$x\mapsto rx$ on~$H\leqslant\mathbb{Q}$, then
\[
h(T)=\sum_{p\in\mathbb{P}\cup\{\infty\}}\log^{+}\vert r\vert_p
=\log\max\{\vert a\vert,\vert b\vert\}
\]
where~$r=\frac{a}{b}$ with~$\gcd(a,b)=1$.
\end{proposition}

\begin{proof}
This formula is explained under the assumption that  $H=R_{\mathcal{S}(H)}$
by the following argument.
First, the calculation~\eqref{equation:entropy}
may be performed in the covering space~$\mathbb{A}_{\mathcal{S}(H)}$
because the quotient map
\[
\mathbb{A}_{\mathcal{S}(H)}
\rightarrow
\widehat{R_{\mathcal{S}(H)}}
\]
is a local isometry
(we refer to~\cite[Prop.~4.7]{ELW} for the details
of this general lifting principle).
Secondly, a small metric ball around~$0$
in the covering space~$\mathbb{A}_{\mathcal{S}(H)}$
takes the form
\[
(-\epsilon,\epsilon)\times\prod_{p\in\mathcal{S}(H)}p^{n_p}\mathbb Z_p,
\]
and the action~$\widetilde{T_r}$ of multiplication by~$r^{-1}$ on each
of the coordinates~$\mathbb R$ or~$\mathbb Q_p$ gives
\[
\bigcap_{j=0}^{n-1}T_{r}^{-j}(-\epsilon,\epsilon)=
\begin{cases}
(-\epsilon,\epsilon)&\mbox{if }\vert r\vert\le 1;\\
(-\epsilon\vert r\vert^{-(n-1)},\epsilon\vert r\vert^{-(n-1)})&
\mbox{if }\vert r\vert>1,
\end{cases}
\]
on the quasi-factor~$\mathbb R$, and
\[
\bigcap_{j=0}^{n-1}T_{r}^{-j}p^{n_p}\mathbb Z_p=
\begin{cases}
p^{n_p}\mathbb Z_p&\mbox{if }\vert r\vert_p\le 1;\\
\vert r\vert_p^{-(n-1)}p^{n_p}\mathbb Z_p&
\mbox{if }\vert r\vert>1,
\end{cases}
\]
showing the claimed formula.

In the general case an argument is needed to ensure that
the entropy is not changed by passing from~$R_{\mathcal{S}(H)}$ to the
direct limit~$H$, and
we refer to~\cite[Prop.~3.1]{MR961739} for the details.
\end{proof}

\subsection{Equivalence relations for one-solenoids}

We now consider how the equivalence relations from Definition~\ref{definition:handzetaequivalence} behave for
the subspace~$\mathcal{G}_1\subset\mathcal{G}$
consisting of
pairs~$(G,T)$ where~$T:G\to G$ is an
automorphism of a one-solenoid.
Thus we may assume that~$G$ is dual to a
group~$H=H((k_p))$, and~$T=T_{a/b}$ is dual to
the map~$x\mapsto\frac{a}{b}x$
with~$\gcd(a,b)=1$
on~$H$,
with the property that any prime~$p$
dividing~$ab$ has~$k_p=\infty$.

\begin{example}
It is clear that the entropy and the zeta function
do not together determine an element of~$\mathcal{G}_1$.
\begin{enumerate}
\item[\rm(1)] We have~$(\widehat{\mathbb Z[\frac12]},T_2)\sim_{\zeta}
(\widehat{\mathbb Z[\frac12]},T_{1/2})$ and~$(\widehat{\mathbb Z[\frac12]},T_2)\sim_{h}
(\widehat{\mathbb Z[\frac12]},T_{1/2})$.
\item[\rm(2)] By varying the group instead of the map, much
    larger joint equivalence classes may be found. As
    pointed out by Miles~\cite{MR2441142}, a consequence of
    Proposition~\ref{fixedpointformula} and
    Proposition~\ref{lemma:abramovformula} is that the set
\[
\left\{(G,T)\in\mathcal{G}_1\mid\zeta_{T}(z)=
\textstyle\frac{1-z}{1-2z}, h(T)=\log2\right\}
\]
is uncountable. To see
this, let~$\{S_{\lambda}\}_{\lambda\in\Lambda}$
be an uncountable set of infinite subsets of~$\mathbb P$, all containing~$2$,
with the
property that~$\vert S_{\lambda}\vartriangle
S_{\nu}\vert=\vert\mathbb N\vert$ for all~$\lambda\neq\nu$.
Associate to each~$S_{\lambda}$ the subgroup~$H_{\lambda}=H((k_p^{(\lambda)}))$
where
\[
k_p^{(\lambda)}=\begin{cases}
0&\mbox{if }p\notin S_{\lambda},\\
1&\mbox{if }p\in S_{\lambda}\setminus\{2\},\\
\infty&\mbox{if }p=2.
\end{cases}
\]
Then the automorphism dual to~$x\mapsto 2x$ on~$H_{\lambda}$
has zeta function~$\frac{1-z}{1-2z}$, and by
Theorem~\ref{theorem:baer} these are all algebraically, and
hence topologically, distinct systems.
\end{enumerate}
\end{example}

\begin{lemma}\label{lemma:zetadeterminesPinfinity}
If~$(G,T)\sim_{\zeta}(G',T')$ then~$\mathcal{S}(G)=\mathcal{S}(G')$.
\end{lemma}

\begin{proof}
Assume that~$(G,T)\in\mathcal{G}_1$ with~$T=T_{a/b}$ and~$G$
dual to~$H((k_p))$.
We claim that
\begin{equation}\label{equation:zetadeterminesPinfinity}
\{p\in\mathbb P\mid p\divides\fix_T(n)\mbox{ for some }n\ge1\}
=
\{p\in\mathbb P\mid p\notin\mathcal{S}(G)\mbox{ and }p\notdivides ab\}.
\end{equation}
Since by hypothesis~$T$ is an automorphism, we have
\[
p\divides ab\implies p\in\mathcal{S}(G)
\]
so~\eqref{equation:zetadeterminesPinfinity} gives the lemma.
To see~\eqref{equation:zetadeterminesPinfinity}, notice that
by Proposition~\ref{fixedpointformula}
no prime in~$\mathcal{S}(G)$
can divide any~$\fix_T(n)$, and any prime~$p\notin\mathcal{S}(G)$
will divide~$a^{p-1}-b^{p-1}$ and so will divide some~$\fix_T(n)$.
\end{proof}

\begin{problem}
The case of endomorphisms is slightly different, because
for example~$(\widehat{\mathbb Z},T_2)\sim_{\zeta}(\widehat{\mathbb Z[\frac12]},T_{2})
\sim_{\zeta}(\widehat{\mathbb Z[\frac12]},T_{1/2})$. Formulate
a version of Lemma~\ref{lemma:zetadeterminesPinfinity} for
endomorphisms of one-solenoids.
\end{problem}

The case of subrings, or equivalently of~$\mathcal{S}$-integer subgroups of~$\mathbb Q$,
has distinctive features. Let~$\overline{\mathcal{G}_1}$ denote the collection
of pairs~$(G,T)\in\mathcal{G}_1$ with the property that~$\widehat{G}$ is
a subring of~$\mathbb Q$. This means that~$k_p(G)$ is~$0$ or~$\infty$ for
any prime~$p$.

\begin{proposition}\label{proposition:sizeofequivalenceclassesinonesolenoids}
On the space~$\overline{\mathcal{G}_1}$,
\begin{enumerate}
\item[\rm(1)] $\sim_h$ has uncountable equivalence classes;
\item[\rm(2)] $\sim_{\zeta}$ has countable equivalence classes;
\item[\rm(3)] the joint relation~$\sim_{h}$ and~$\sim_{\zeta}$ has
finite equivalence classes.
\end{enumerate}
\end{proposition}

\begin{proof}
The pair~$(\widehat{\mathbb Z[\frac{1}{ab}]},T_{a/b})$
has entropy~$\log\max\{\vert a\vert,\vert b\vert\}$,
and by Abramov's formula for any subset~$\mathcal{S}\subset\mathbb P\setminus\{p\mid p\divides ab\}$
of primes the pair~$(\widehat{R_\mathcal{S}[\frac{1}{ab}]},T_{a/b})$
has the same entropy, showing~(1).

By Lemma~\ref{lemma:zetadeterminesPinfinity}, the~$\sim_{\zeta}$
equivalence class determines the set of primes of infinite
height, and so the only parameter that can change
is the rational~$\frac{a}{b}$ defining the map~$T_{a/b}$,
and there are only countably many of these.

For~(3), notice that if~$(G_2,T_{a/b})$ has entropy~$h$ then
\[
\max\{\vert a\vert,\vert b\vert\}
\le\exp h,
\]
so there are only finitely many choices for the rational~$r$ defining
the map~$T_r$ with a given entropy. Lemma~\ref{lemma:zetadeterminesPinfinity}
shows that there is no further choice for a fixed
zeta function.
\end{proof}

\begin{example}\label{example:diversity}
Some of the diversity possible in Proposition~\ref{proposition:sizeofequivalenceclassesinonesolenoids}
is illustrated via simple examples.
\begin{enumerate}
\item[\rm(1)] For any~$r\in\mathbb Q\setminus\{0,\pm1\}$
the system~$(\widehat{\mathbb Q},T_r)$ has
zeta function~$\frac{1}{1-z}$, showing that the~$\sim_{\zeta}$
equivalence class may be infinite.
\item[\rm(2)] The~$\sim_{\zeta}$ equivalence class may be infinite
in a less degenerate way.
Let~$G$ be the group dual to~$\mathbb Z_{(3)}$,
the subgroup obtained by inverting all the primes
except~$3$.
Then Proposition~\ref{fixedpointformula} shows that
\[
\fix_{T_r}(n)=\vert r^n-1\vert_3^{-1}
\]
by the product formula
\begin{equation}\label{productformula}
\prod_{p\in\mathbb P\cup\{\infty\}}\vert t\vert_p
=
\vert t\vert
\negmedspace
\times
\negmedspace\negmedspace\negmedspace
\prod_{p\in\mathbb P:\vert t\vert_p\neq1}\vert t\vert_p
=
1
\end{equation}
for all~$t\in\mathbb Q\setminus\{0\}$.
It follows that~$(G,T_{2})\sim_{\zeta}(G,T_{r})$
if~$r=\frac{a}{b}$ where~$(a,b)$ is
of the form
\[
\left(
\textstyle\frac{9k+3m}{2}+2,\frac{9k-3m}{2}+1
\right)
\]
or
\[
\left(
\textstyle\frac{9k+3m}{2}+4,\frac{9k-3m}{2}+2
\right)
\]
for integers~$m,k$ of the same parity chosen with~$a,b$ coprime
and positive, or of the form
\[
\left(
\textstyle\frac{9k+3m+5}{2},\frac{9k-3m+1}{2}
\right)
\]
or
\[
\left(
\textstyle\frac{9k+3m+7}{2},\frac{9k-3m+5}{2}
\right)
\]
for integers~$m,k$ of opposite parity chosen with~$a,b$ coprime
and positive.
\item[\rm(3)] The~$\sim_{\zeta}$ equivalence class may be finite.
For example, if~$(G,T)\in\overline{\mathcal{G}_1}$ has
zeta function~$\frac{1-z}{1-2z}$, then we claim that~$(G,T)$
can only be~$(\widehat{\mathbb Z[\frac12]},T_2)$
or~$(\widehat{\mathbb Z[\frac12]},T_{1/2})$.
To see this, notice first that these both have
the claimed zeta function, and apply Lemma~\ref{lemma:zetadeterminesPinfinity}
to deduce that any
element of~$\overline{\mathcal{G}_1}$ with the
same zeta function has the form~$(\widehat{\mathbb Z[\frac12]},T_{a/b})$.
Since the map~$T_{a/b}$ is an automorphism, the only prime
dividing~$ab$ is~$2$, so~$\frac{a}{b}$ is~$\pm 2^k$ for some~$k\in\mathbb Z$
and we may use Proposition~\ref{fixedpointformula}
to calculate
\[
\zeta_{T}(z)=
\begin{cases}
\frac{1-z}{1-2^kz}&\mbox{if }T=T_{2^k};\\
\frac{1+z}{1-2^kz}&\mbox{if }T=T_{-2^k}.
\end{cases}
\]
\end{enumerate}
\end{example}

The constraints seen in
Example~\ref{example:diversity}(3) hold more generally
for systems with the property that~$\vert\mathcal{S}(\widehat{G})\vert$
is finite.

\begin{proposition}
Let~$(G_1,T_r),(G_2,T_s)\in\mathcal{G}_1$
have~$\vert\mathcal{S}(\widehat{G_1})\vert<\infty$
and
\[
(G_1,T_r)\sim_{\zeta}(G_2,T_s).
\]
Then~$r=s$
or~$r=s^{-1}$.
\end{proposition}

\begin{proof}
First notice
that~$\mathcal{S}(\widehat{G_2})=\mathcal{S}(\widehat{G_1})$ by
Lemma~\ref{lemma:zetadeterminesPinfinity}. Denote this common
set of primes by~$\mathcal{S}$.
Write~$r=\frac{a_1}{b_1}$,~$s=\frac{a_2}{b_2}$
with~$\gcd(a_i,b_i)=1$ and~$a_i> 0$ for~$i=1,2$.
Since~$\zeta_{T_r}=\zeta_{T_r^{-1}}$
and~$\zeta_{T_s}=\zeta_{T_s^{-1}}$, without loss of generality
we may assume that~$a_i>\vert b_i\vert$ for~$i=1,2$. Then,
using Proposition~\ref{fixedpointformula}
(see~\cite[Th~6.1]{MR1461206} for the details), it follows that
\[
\log a_1=\lim_{n\to\infty}\frac{1}{n}\log\fix_{T_r}(n)
\mbox{ and }
\log a_2=\lim_{n\to\infty}\frac{1}{n}\log\fix_{T_s}(n),
\]
so that $a_1=a_2$. Let $a=a_1=a_2$,
\[
\mathcal{T}_i=\{p\in\mathcal{S}\mid\vert a\vert_p<\vert b_i\vert_p\}
\mbox{ and }
\mathcal{T}'_i=\{p\in\mathcal{S}\mid\vert a\vert_p>\vert b_i\vert_p\},
\]
for~$i=1,2$.
Notice that since~$\gcd(a,b_i)=1$,~$\vert b_i\vert_p=1$ for
all~$p\in \mathcal{T}_i$, and~$\vert a\vert_p=1$
for all~$p\in \mathcal{T}'_i$,~$i=1,2$.
Hence
\begin{equation}\label{like_a_natural_man}
\prod_{p\in\mathcal{T}_i\cup\mathcal{T}'_i}\left\vert\left({\frac{a}{b_i}}\right)^n-1\right\vert_p
 =
\prod_{p\in\mathcal{T}'_i}\left\vert\left(\frac{a}{b_i}\right)^n-1\right\vert_p
 =
\prod_{p\in\mathcal{T}'_i}\vert b_i\vert_p^{-n}
=
\vert b_i\vert^n,
\end{equation}
where the last equality follows by the
Artin product formula, as~$\vert b_i\vert_p=1$ for
all~$p\in\mathbb P\setminus\mathcal{T}'_i$.

For each~$p\in\mathcal{S}\setminus(\mathcal{T}_i\cup\mathcal{T}'_i)$,~$\vert\frac{a}{b_i}\vert_p=1$,
that is,~$\frac{a}{b_i}$ is a unit in the valuation
ring~$\mathbb{Z}_{(p)}$. Let~$m_i(p)$ denote the
multiplicative order of~$\frac{a}{b_i}$ in
the residue field~$\mathbb{Z}_{(p)}/(p)$, and
note that if~$m_i(p)\notdivides n$,
then~$\vert(\frac{a}{b_i})^n-1\vert_p=1$.
Therefore, whenever~$n$ is coprime to
\[
m=\prod_{i=1,2}\prod_{p\in \mathcal{S}\setminus(\mathcal{T}_i\cup\mathcal{T}'_i)}m_i(p),
\]
for both~$i=1$ and~$i=2$ we have
\begin{equation}\label{i_refuse_to_listen}
\prod_{p\in\mathcal{S}\setminus(\mathcal{T}_i\cup\mathcal{T}'_i)}
\left\vert\left(\frac{a}{b_i}\right)^n-1\right\vert_p
=1.
\end{equation}
Substituting~\eqref{like_a_natural_man}
and~\eqref{i_refuse_to_listen} into the
formula given by Proposition~\ref{fixedpointformula} yields
\[
\fix_{T_r}(n)
=
\vert a^n-b_1^n\vert
\mbox{ and }
\fix_{T_s}(n)
=
\vert a^n-b_2^n\vert,
\]
whenever~$\gcd(n,m)=1$.
Since~$a>0$ and~$a>\vert b_1\vert$,
it follows that~$b_1=\pm b_2$.
\end{proof}

\begin{corollary}
The~$\sim_{\zeta}$ equivalence class of~$(G,T_{a/b})$ in the subset
of~$\mathcal{G}_1$ with~$\mathcal{S}(\widehat{G})$ finite
has cardinality~$2^{\omega(a)}+2^{\omega(b)}$, where
as usual~$\omega(k)$ denotes the number of
distinct prime divisors of~$k$.
\end{corollary}

%
%

\section{Counting closed orbits}

The most transparent combinatorial invariant of topological conjugacy
is the count of periodic orbits, which may be captured
using generating functions like the dynamical
zeta function.
For a class of ergodic group automorphisms with
polynomially bounded growth in the number of closed
orbits, a more natural generating function
is given by the orbit Dirichlet series~\cite{MR2550149}.
A consequence of the structure theorem for
expansive automorphisms of connected groups due
to Lawton~\cite{MR0391051} is the following.

\begin{theorem}[Lawton~\cite{MR0391051}]
If~$T:G\to G$ is an ergodic expansive automorphism
of a compact connected group, then~$\zeta_T(z)$ is
a rational function of~$z$.
\end{theorem}

It is clear that not every function can be a zeta function of a
map -- in particular the coefficients
in~\eqref{equation:zetafunction} must be non-negative. In fact
more is true, and it is shown in~\cite{MR1873399} that a
function~$\zeta(z)=\exp\sum_{n\ge1}\textstyle\frac{a_n}{n}z^n$
is the dynamical zeta function of some map if and only if
\begin{equation}\label{equation:puriward}
0\le\sum_{d\vert n}\mu\(\textstyle\frac{n}{d}\)a_d\equiv0\pmod{n}
\end{equation}
for all~$n\ge1$.
A beautiful remark of Windsor is that
the same condition is equivalent to being the dynamical
zeta function of a~$C^{\infty}$ diffeomorphism of the~$2$-torus~\cite{MR2422026}.
Thus, for example,
\[
\frac{{\rm e}^{-z^2}}{(1-2z)}
=\exp
\(
2z+z^2+\frac{2^3}{3}z^3+\frac{2^4}{4}z^4+\frac{2^5}{5}z^5+\cdots
\)
\]
has non-negative coefficients but cannot be a dynamical zeta
function of any map. England and Smyth~\cite{MR0307280}
characterized in combinatorial terms the property of being the
dynamical zeta function of an automorphism of a one-dimensional
solenoid, and Moss~\cite{moss} considered the more general
question of when a function could be the dynamical zeta
function of a group automorphism. Clearly the non-negativity
and congruence condition in~\eqref{equation:puriward} is not
sufficient, since the sequence~$\fix_T=(\fix_T(n))_{n\ge1}$ for
a group automorphism must be a divisibility sequence. Moss
showed that adding this further condition is also not
sufficient, even for linear recurrence sequences.

\begin{lemma}[Moss~\cite{moss}]\label{lemma:moss}
The function~$f(z)=\frac{1}{(1-z)(1-z^5)}$
is the dynamical zeta function of the
permutation~$\tau=(1)(23456)$ on the set~$\{1,2,3,4,5,6\}$,
which has the property that~$\fix_{\tau}$ is a divisibility sequence,
but~$f$ is not the zeta function of any group automorphism.
\end{lemma}

\begin{proof} The given permutation~$\tau$
has
\[
(\fix_{\tau})=(1,1,1,1,6,1,1,1,1,6,\ldots),
\]
which is a divisibility sequence satisfying the
linear recurrence relation
\[
u_{n+5}=u_n
\]
for~$n\ge1$ with the
initial conditions
\[
u_1=u_2=u_3=u_4=1,u_5=6.
\]
To see that~$f$ cannot be the dynamical zeta function
of an abelian group automorphism, notice that if there is
a group automorphism~$T:G\to G$ with~$\zeta_{T}(z)=f(z)$
then there is such a realization with~$\vert G\vert=6$.
In this case~$T$ is an automorphism of~$G$ with~$T^5=I$, the
identity. If some~$g\in G\setminus\{0\}$ has orbit
under~$T$ of cardinality less than~$5$, then
there are integers~$0\le m<n\le 4$ with~$T^mx=T^nx$,
which implies that~$T^{n-m}x=x$ and so~$x=0$
since~$\vert n-m\vert\le4$.
It follows that any non-identity element of~$G$ has
an orbit of length~$5$ under the automorphism~$T$,
and so in particular every non-identity element of~$G$
has the same order, which is impossible.
\end{proof}

England and Smyth~\cite{MR0307280} showed that
\[
f(z)=\exp\sum_{n\ge1}\frac{a_n}{n}z^n
\]
is the dynamical zeta function of a group
automorphism dual to~$x\mapsto\frac{m}{n}x$ on
a one-solenoid if and only if
\begin{equation}\label{Turned cartwheels 'cross the floor}
a_k\divides n^k-m^k\mbox{ for all~$k\ge1$}
\end{equation}
and
\begin{equation}\label{I was feeling kinda seasick}
\gcd\(a_k,\textstyle\frac{n^{\ell}-m^{\ell}}{a_{\ell}}\)=1\mbox{ for all~$k\neq\ell$}.
\end{equation}
They also gave an example (which does not seem to be correct)
of a group automorphism with irrational zeta function
for an automorphism of a one-solenoid;
there are now several ways to see that these must exist.
In~\cite{MR1458718} it is shown that there are
uncountably many distinct zeta functions of
automorphisms dual to~$x\mapsto2x$ on a one-solenoid
(and hence most in cardinality are irrational functions); Everest, Stangoe
and the third author~\cite[Lem.~4.1]{MR2180241}
showed that the map dual to~$x\mapsto2x$ on~$\mathbb Z[\frac{1}{6}]$
has a natural boundary at~$\vert z\vert=\frac12$. In
order to relate the characterization from~\cite{MR0307280}
to the~$S$-integer constructions of Chothi, Everest
{\it et al.}~\cite{MR1461206}
outlined in Proposition~\ref{fixedpointformula},
we show that they are equivalent.

\begin{lemma}
Given coprime integers~$m,n$, a sequence~$(a_n)$ of positive integers satisfies~\eqref{Turned cartwheels 'cross the floor}
and~\eqref{I was feeling kinda seasick} if and only if there is a
subset of the primes~$\mathcal{S}\subset\mathbb P$ for which
\begin{equation}\label{But the crowd called out for more}
a_k=\vert n^k-m^k\vert\prod_{p\in \mathcal{S}}\vert n^k-m^k\vert_p
\end{equation}
for all~$k\ge1$.
\end{lemma}

\begin{proof}
Clearly the existence of a set~$\mathcal{S}\subset\mathbb P$ with~\eqref{But the crowd called out for more}
implies~\eqref{Turned cartwheels 'cross the floor}
and~\eqref{I was feeling kinda seasick}.

Assume~\eqref{Turned cartwheels 'cross the floor}
and~\eqref{I was feeling kinda seasick}, let
\[
\mathcal{S}=\{p\in\mathbb P\mid p\divides\textstyle\frac{n^{\ell}-m^{\ell}}{a_{\ell}}\mbox{ for some }\ell\ge1\},
\]
and write~$b_k=\vert n^k-m^k\vert\prod_{p\in \mathcal{S}}\vert n^k-m^k\vert_p$. If some prime power~$p^e$
divides some~$a_k$, then~$p^e\vert n^k-m^k$ so by~\eqref{I was feeling kinda seasick}~$p^e\divides b_k$,
and so~$a_k\divides b_k$ for all~$k\ge1$.
If some prime power~$p^e$ divides some~$b_k$,
then~$p^e\divides n^k-m^k$ and~$p\notin \mathcal{S}$ and hence~$p^e\divides a_k$. Thus~$b_k\divides a_k$ for all~$k\ge1$.
\end{proof}

Miles gave a general
characterization for finite entropy automorphisms on finite-dimensional
groups. In order to state this, recall from Weil~\cite{MR0234930}
that a global field (or~$\mathbb{A}$-field) is
an algebraic number field or a function field
of transcendence degree one over a finite field.
For a global field~$K$ we write~$\mathbb{P}(K)$ for
the set of places of~$K$, and~$\mathbb{P}_{\infty}(K)$ for
the set of infinite places of~$K$

\begin{theorem}[Miles~\cite{MR2441142}]
If~$T:G\to G$ is an ergodic
automorphism of a finite-dimensional
compact abelian group with finite entropy, then
there exist global fields~$K_1,\dots,K_n$, sets of finite places~$P_i\subset
\mathbb{P}(K_i)$ and
elements~$\xi_i\in K_i$, no one of
which is a root of unity for~$i=1,\dots,n$, such that
\[
\fix_{T}(k)=\prod_{i=1}^{n}\prod_{v\in P_i}
\vert\xi_i^k-1\vert_{v}^{-1}.
\]
\end{theorem}

We now turn to another range problem, concerned with the growth
in orbits. This may be thought of as a combinatorial analogue of
Lehmer's problem.

\begin{theorem}[Ward~\cite{MR2085157}]\label{theorem:allgrowthrates}
For any~$C\in[0,\infty]$ there is a compact group automorphism~$T:G\to G$
with
\[
\frac{1}{n}\log\fix_{T}(n)\longrightarrow C
\]
as~$n\to\infty$.
\end{theorem}

\subsection{Sets of primes and constructions in one-solenoids}

The examples constructed
in
Theorem~\ref{theorem:allgrowthrates}
are not really satisfactory, since the groups used are
zero-dimensional and (worse) the automorphisms are not ergodic.
As a result it is natural to ask about the range of various
invariants for ergodic automorphisms of connected groups.
For rapid growth and a crude approximation, it is possible to
exhibit growth in periodic orbits in a controlled way on
connected groups, as illustrated in the next result
(the proof of which we postpone briefly).

\begin{theorem}\label{theorem:newaddedconnectedrates}
Let~$\theta=(\theta_n)$ be a sequence of positive
integers with
\[
\frac{n}{\log\theta_n}\longrightarrow 0
\]
as~$n\to\infty$. Then there is an automorphism~$T:X\to X$
of a connected compact group with
\[
\frac{\log\fix_T(n)}{\log\theta_n}\longrightarrow1
\]
as~$n\to\infty$.
\end{theorem}

With more
effort, taking advantage of known estimates for the size
of the primitive part of linear recurrence sequence, it is
likely that this result can be improved, but it is unlikely
that any such approach will lead to an (unexpected)
positive answer to the following.

\begin{problem}
Does Theorem~\ref{theorem:allgrowthrates} still hold
if the compact groups are required to be connected?
\end{problem}

It is expected that this reduces directly to Lehmer's problem,
and is as a result equally intractable. A refined version (in
light of the expected negative answer) is the following;
background and partial results in this direction may be
found in some work of the third author~\cite{MR1619569},~\cite{MR1458718},
exposing in particular some of the Diophantine
problems that arise.

\begin{problem}
Characterize the set
\[
\left\{
\limsup_{n\to\infty}\frac{1}{n}\log\fix_T(n)\mid
(G,T)\in\mathcal{G}\mbox{ and }G\mbox{ is connected}
\right\}
\]
of real numbers.
\end{problem}

As we have seen in Proposition~\ref{fixedpointformula},
orbit-growth questions on the space~$\mathcal{G}_1$
essentially amount to questions about rational numbers
and sets of primes.
In this section we will give an overview
of what is known in this setting specifically
for the map dual to~$x\mapsto 2x$ on subrings
of~$\mathbb Q$. Thus the systems are parameterised by
subsets of~$\mathbb P$, and Figure~\ref{figure}
gives a guide to the results assembled here on
the Hasse diagram of subsets of~$\mathbb P$, or equivalently
of subrings of~$\mathbb Q$.

\begin{figure}[htbp]\begin{center}
\psfrag{rat}{$\zeta_T$ rational}\psfrag{irrat}{$\zeta_T$ has natural}
\psfrag{irrat2}{boundary~$\vert z\vert={\rm e}^{-h}$}\psfrag{trivial}{one orbit}
\psfrag{dirichlet}{D-rational orbit}
\psfrag{dirichlet2}{Dirichlet series}
\psfrag{1}{continuum 1}\psfrag{2}{continuum 2}\psfrag{3}{continuum 3}
\psfrag{Z}{$\mathbb Z$}\psfrag{Z2}{$\mathbb Z[\frac12]$}\psfrag{Z3}{$\mathbb Z[\frac13]$}
\psfrag{Z5}{$\mathbb Z[\frac15]$}\psfrag{Z23}{$\mathbb Z[\frac16]$}\psfrag{Z35}{$\mathbb Z[\frac{1}{15}]$}\psfrag{Z25}{$\mathbb Z[\frac{1}{10}]$}
\psfrag{Q}{$\mathbb Q$}\psfrag{Q2}{$\mathbb Q_{(2)}$}\psfrag{Q3}{$\mathbb Q_{(3)}$}
\psfrag{Q5}{$\mathbb Q_{(5)}$}\psfrag{Q23}{$\mathbb Q_{(2,3)}$}
\psfrag{Q25}{$\mathbb Q_{(2,5)}$}\psfrag{Q35}{$\mathbb Q_{(3,5)}$}
\psfrag{k}{$k$}\psfrag{delta}{$\delta$}
\psfrag{k2}{$k$}\psfrag{r}{$r$}\psfrag{kappa}{$\kappa$}
\scalebox{1.0}[1.0]{\includegraphics{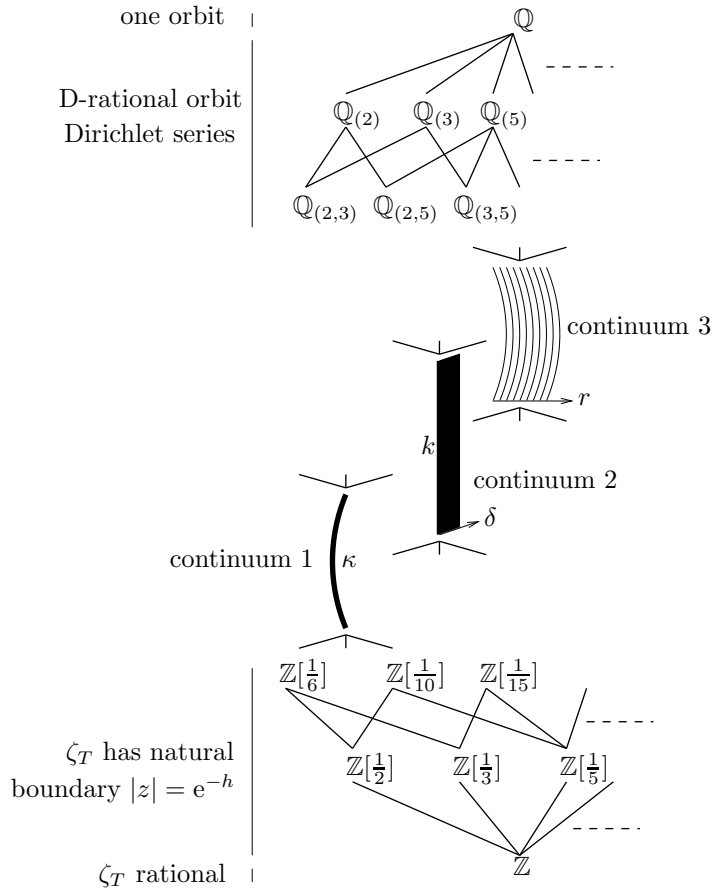}}
\caption{\label{figure}The Hasse diagram of subrings of~$\mathbb Q$.}
\end{center}
\end{figure}

At the bottom of Figure~\ref{figure} the compact group
is the circle, and strictly speaking this has no ergodic
automorphisms. However, as seen in Example~\ref{fixedpointsexamples},
if we permit the endomorphisms~$x\mapsto ax$ (or invert
only those primes appearing in~$a$) then the
zeta functions arising are rational.

The first observation one can make about the lower
part of Figure~\ref{figure} (subrings of~$\mathbb Q$
in which~$k_p=\infty$ for only finitely many primes~$p$)
concerns the exponential growth rate of periodic points.
If~$S\subset\mathbb P$ has~$\vert S\vert<\infty$, then
there are constants~$A,B>0$ with
\begin{equation}\label{equation:basicestimate}
\frac{A}{n^B}\le\prod_{p\in S}\vert a^n-b^n\vert_p\le1
\end{equation}
for all~$n\ge1$
(see~\cite[Th~6.1]{MR1461206} for the details).
A simple consequence of
Proposition~\eqref{fixedpointformula} and the estimate~\eqref{equation:basicestimate}
gives the construction needed for Theorem~\ref{theorem:newaddedconnectedrates}.

\begin{proof}[Proof of Theorem~\ref{theorem:newaddedconnectedrates}]
Define a sequence~$(S_k)$ of sets of primes by
\[
S_k=\{p\in\mathbb P\mid
p\divides 2^n-1\mbox{ for some }n<k\}\cup\{2\}.
\]
Thus~$S_1=\{2\}$,~$S_2=\{2,3\}$,~$S_3=\{2,3,7\}$, and so on,
and let~$T_k:X_k\to X_k$ be the automorphism
dual to~$x\mapsto 2x$ on~$\mathbb Z[\frac{1}{p}:p\in S_k]$.
By Proposition~\ref{fixedpointformula}
and~\eqref{equation:basicestimate}
it follows that~$\fix_{T_k}(n)=1$ for~$n<k$
and~$\lim_{n\to\infty}(1/n)\log\fix_{T_k}(n)=\log2$.
It may be helpful to
have in mind that the sequence of
fixed point counts for these systems begins
as follows:
\begin{align*}
\fix_{T_1}&=(1,3,7,15,31,\dots)\\
\fix_{T_2}&=(1,1,7,5,31,\dots)\\
\fix_{T_3}&=(1,1,1,5,31,\dots)\\
\fix_{T_4}&=(1,1,1,1,31,\dots),
\end{align*}
and so on.
We now inductively construct an infinite-dimensional
solenoid~$X$ together with an automorphism~$T$ as follows,
starting at~$2$ for convenience.

At the first stage, let~$n_2=\lceil\log_3\theta_2\rceil$
and let~$(X^{(2)},T^{(2)})$ be the~$n_2$-fold Cartesian
product of copies of~$(X_1,T_1)$.
By construction,
\begin{equation}\label{threeacttragedy}
1\le\frac{\fix_{T^{(2)}}(2)}{\theta_2}\le3=\fix_{T_1}(2).
\end{equation}
At the next stage, let~$n_3=\lceil\log_{7}\theta_3\rceil-n_2$ (and assume,
as we may, that this is non-negative by amending the early terms of the
sequence~$\theta$), and let~$(X^{(3)},T^{(3)})$ be the
product of~$(X^{(2)},T^{(2)})$
and the~$n_3$-fold
Cartesian product of copies of~$(X_2,T_2)$.
By construction, we retain~\eqref{threeacttragedy}
but also have
\begin{equation*}
1\le\frac{\fix_{T^{(3)}}(3)}{\theta_3}\le7=\fix_{T_2}(3).
\end{equation*}
We continue in this way,
multiplying at the~$n$th stage by a number of
copies of~$(X_k,T_k)$ to approximate as well
as possible the number of points of period~$n$ sought,
eventually producing an
infinite product of one-dimensional solenoids~$T$
with an automorphism~$T$ so that
\[
1\le\frac{\fix_{T}(n)}{\theta_n}\le\fix_{T_{n-1}}(n)\le2^n.
\]
Thus
\[
1\le\frac{\log\fix_{T}(n)}{\log\theta_n}\le1+\frac{n\log2}{\log\theta_n},
\]
giving the result by the growth-rate assumption on~$\theta$.
\end{proof}

\begin{example}\label{example:stangoe}
The map~$T$ dual to~$x\mapsto2x$ on~$\mathbb Z[\frac16]$ has
\[
\zeta_{T}(z)=\exp\sum_{n\ge1}\frac{\vert 2^n-1\vert\vert 2^n-1\vert_3}{n}z^n.
\]
The radius of convergence of~$\zeta_T$ is~$\frac{1}{2}$
by~\eqref{equation:basicestimate},
and a simple argument from~\cite[Lem.~4.1]{MR2180241} shows that
the series defining~$\zeta_T(z)$ has a zero at all points
of the form~$\frac{1}{2}{\rm e}^{2\pi{\rm i}j/3^r}$ for~$j\in\mathbb Z$
and~$r\ge1$, so~$\vert z\vert=\frac12$ is a natural boundary.
\end{example}

The appearance of a natural boundary for the zeta function
at~$\frac12=\exp(-h(T))$ prevents
an analysis of~$\pi_T$
and~$\mertens_T$ using analytic methods.
On the other hand, the dense set of
singular points
on the circle~$\vert z\vert=\frac12$
form a discrete set in~$\mathbb C_3$ (the completion
of the algebraic closure of~$\mathbb Q_3$), raising
the following question.

\begin{problem}\label{moonshine1}
For the map in Example~\ref{example:stangoe}
develop a notion of a dynamical
zeta function~$\zeta_T:\mathbb C_3\rightarrow\mathbb C_3$
in such a way that it has a meromorphic extension
beyond its radius of convergence, and
exploit this to give analytic or Tauberian proofs
of statements like
\[
\mertens_T(X)=\textstyle\frac58\log N+C_T+\bigo(1/N).
\]
\end{problem}

There are several difficulties involved here,
starting with the fact that the formal power
series~$\exp$ on~$\mathbb Q_3$ has radius
of convergence~$3^{-1/2}<1$; we refer
to Koblitz~\cite{MR754003} for the background needed
to address this problem.

\begin{problem}
For a system in~$\mathcal{G}_{1}$ for which the set of primes
of infinite height is finite, show that if the zeta function is
irrational, then it has natural boundary at~${\rm e}^{-h}$.
\end{problem}

More generally, there is some reason to suspect a stronger
result (see~\cite{bmw} for more on the methods involved
here, and some partial results). In the disconnected case it is
known that the zeta function is typically
non-algebraic~\cite{MR1702897}.

\begin{problem}[Bell, Miles and Ward~\cite{bmw}]
Is there a P{\'o}lya--Carlson dichotomy for
zeta functions of automorphisms of compact
groups? That is, is it true that such a zeta function is either
rational or has a natural boundary at its circle of
convergence?
\end{problem}

Results of~\cite{MR2339472} give precise information about the
systems at the lower end of Figure~\ref{figure}, and in the
special case of one-solenoids they take the following form.

\begin{theorem}[Everest, Miles, Stevens and Ward~\cite{MR2339472}]\label{theorem:natural}
If~$T$ is an automorphism in~$\mathcal{G}_1$
with a finite set of primes of infinite
height, then there are constants~$k_T$
in~$\mathbb Q$ and~$C_T>0$
with
\[
\mertens_T(n)=k_T\log N+C_T+\bigo(1/N).
\]
\end{theorem}

\begin{example}
If~$T$ is the automorphism dual to~$x\mapsto 2x$ on~$\mathbb Z[\frac{1}{21}]$
then
\[
\mertens_T(N)=\textstyle\frac{269}{576}\log N+C+\bigo(1/N).
\]
\end{example}

Problem~\ref{moonshine1} also arises in this
setting, taking the following form.

\begin{problem}\label{moonshine2}
For systems in~$\mathcal{G}_{1}$ for which the set of primes of
infinite height is finite, develop a notion of dynamical zeta
function on the space~$\prod_{p:k_p=\infty}\mathbb C_p$ so that
it has a meromorphic extension beyond its radius of
convergence, and exploit this to give an analytic proof of
Theorem~\ref{theorem:natural}.
\end{problem}

At the opposite extreme, the system corresponding to
the subring~$\mathbb Q$ has one closed orbit of length~$1$
by the product formula. If the set of primes of infinite height is
co-finite, then the same
basic $p$-adic estimate~\eqref{equation:basicestimate}
shows that
\begin{equation}\label{pp_polynomial_bound}
\fix_{T}(n)\le n^A
\end{equation}
for some~$A>0$.
This polynomial bound makes it natural to study the
orbit growth using the orbit Dirichlet series
\[
\dirichlet_T(s)=\sum_{n\ge1}\frac{\orbit_T(n)}{n^s},
\]
and in~\cite{MR2550149} this is done for a large class
of automorphisms of connected groups where~\eqref{pp_polynomial_bound}
is guaranteed. At the level of generality adopted
in~\cite{MR2550149}, estimates for the growth of~$\pi_{T}(N)$
depend on the abscissa of convergence of~$\dirichlet_T(s)$.
However, if~$T$ is assumed to be an automorphism in~$\mathcal{G}_1$,
then~$\dirichlet_T(s)$ is found to have abscissa of
convergence~$\sigma=0$ (see~\cite[Rmk.~3.7(1)]{MR2550149}),
and the following applies.

\begin{theorem}[Everest, Miles, Stevens and Ward~\cite{MR2550149}]\label{theorem:dirichlet}
If~$T$ is an automorphism in~$\mathcal{G}_1$
with a co-finite set of primes of infinite
height, then there is a finite set~$M\subset\mathbb N$
for which~$\dirichlet_T(s)$ is a rational
function of the variables~$\{s^{-m}\mid m\in M\}$, and
 there is a constant~$C>0$ such that
\[
\pi_{T}(N)
=
C\left(\log N\right)^K+\bigo\left((\log N)^{K-1}\right),
\]
where~$K$ is the order of the pole at $s=0$.
\end{theorem}

The special structure of the orbit Dirichlet series arising
in the co-finite case, which might be called Dirichlet
rationality, raises the possibility of proving
some or all of Theorem~\ref{theorem:dirichlet}
using Tauberian methods. Unfortunately, as the
next example shows, Theorem~\ref{theorem:dirichlet}
does not seem to be provable using Tauberian methods
for a reason related to the natural boundary
phenomenon in Example~\ref{example:stangoe}:
in all but the simplest of settings, the
analytic behavior on the critical line is poor.
The following example, which
is an easy calculation
using Proposition~\ref{fixedpointformula},
is taken from Everest, Miles, Stevens {\it et al.}~\cite[Ex.~4.2]{MR2550149}.

\begin{example}\label{dirichletseriesexample}
Let~$T:G\to G$ be the automorphism dual to the
map~$x\mapsto 2x$ on the subring~$\mathbb Z_{(3)}\cap\mathbb Z_{(5)}\subset\mathbb Q$.
Then
\[
\dirichlet_{T}(z)=1-\frac{1}{2^{z+1}}+\frac{3}{2^{z+1}}
\left(
1-\frac{1}{3^{z+1}}-\frac{1}{2^{z+1}}+\frac{1}{6^{z+1}}\right)
\frac{1}{1-3^{-z}}
\]
\[+
\frac{15}{4^{z+1}}\left(
1-\frac{1}{3^{z+1}}-\frac{1}{5^{z+1}}+\frac{1}{15^{z+1}}
\right)
\frac{1}{(1-3^{-z})(1-5^{-z})}.
\]
\end{example}

The obstacle to using Agmon's Tauberian theorem~\cite{MR0054079} or
even
Perron's Theorem~\cite[Th.~13]{MR0185094} is illustrated
in Example~\ref{dirichletseriesexample}: not only
does~$\dirichlet_T$ have infinitely many singularities
on the critical line, it has singularities that are
arbitrarily close together. On the other hand,
they are located at arithmetically special points,
raising the analogue of Problem~\ref{moonshine2}.

\begin{problem}
Develop a notion of orbit Dirichlet series
with variables in the space~$\prod_{p:k_p=0}\mathbb C_p$,
and develop~$p$-adic versions of Tauberian
theorems to find a meromorphic extension
of the orbit Dirichlet series arising in the
setting of Theorem~\ref{theorem:dirichlet},
and use this to give analytic proofs of
Theorem~\ref{theorem:dirichlet}(1) and~(3).
\end{problem}


Theorems~\ref{theorem:natural} and~\ref{theorem:dirichlet}
give some insight into the range of orbit-growth invariants
across those groups with finite or
co-finite set of primes of infinite height.
Of course most elements of~$\mathcal{G}_1$ fall
into neither the co-finite nor the finite camp,
and new phenomena appear in this large middle ground.
Using delicate constructions of thin sets of
primes, Baier, Jaidee, Stevens {\it et al.}~\cite{bjsw}
find several different continua of different orbit-growth
invariants inside~$\mathcal{G}_1$.

\begin{theorem}[Baier, Jaidee, Stevens and Ward~\cite{bjsw}]
As illustrated in Figure~\ref{figure}, three different
continua of orbit-growth rates exist in~$\mathcal{G}_1$
as follows.
\begin{enumerate}
\item[\rm(1)] For any~$\kappa\in(0,1)$ there is a system~$(G,T)\in\mathcal{G}_1$
with
\[
\mertens_T(N)\sim\kappa\log N
\]
{(see \rm Continuum~1}).
\item[\rm(2)] For any~$\delta\in(0,1)$ and~$k>0$ there is a system~$(G,T)\in\mathcal{G}_1$
with
\[
\mertens_T(N)\sim k(\log N)^{\delta}
\]
(see {\rm Continuum~2}).
\item[\rm(3)] For any~$r\in\mathbb N$ and~$k>0$ there is a system~$(G,T)\in\mathcal{G}_1$
with
\[
\mertens_T(N)\sim k(\log\log N)^r
\]
(see {\rm Continuum~3}).
\end{enumerate}
\end{theorem}

\subsection*{Acknowledgements}

We are grateful to Anthony Flatters for the calculation in
Example~\ref{example:diversity}(2), and to Alex Clark
for pointing out the paper of Markus and Meyer~\cite{MR556887}.

\providecommand{\bysame}{\leavevmode\hbox to3em{\hrulefill}\thinspace}

\end{document}